\documentclass[10pt,notitlepage,twoside,a4paper]{amsart}
 \usepackage{amsfonts}

\usepackage{amsmath,amssymb,enumerate}

\usepackage{epsfig,fancyhdr,color}%,showkeys,amsmidx

\usepackage{amssymb}
\usepackage{amsmath,amsthm}
\usepackage{latexsym}
\usepackage{amscd}
\usepackage{psfrag}
\usepackage{graphicx}
\usepackage[latin1]{inputenc}
\usepackage[all]{xy}
\usepackage[mathcal]{eucal}

\definecolor{NoteColor}{rgb}{1,0,0}

\renewcommand{\textsc}{\textcolor{red}}

\newtheorem{theorem}{\rm\bf Theorem}[section]
\newtheorem{proposition}[theorem]{\rm\bf Proposition}
\newtheorem{lemma}[theorem]{\rm\bf Lemma}
\newtheorem{corollary}[theorem]{\rm\bf Corollary}
\newtheorem*{theorem 1}{\rm\bf Proposition 1}
\newtheorem*{theorem 2}{\rm\bf Proposition 2}

\theoremstyle{definition}

\theoremstyle{remark}

\def\interieur#1{\mathord{\mathop{\kern 0pt #1}\limits^\circ}}

\title[Fenchel-Nielsen coordinates]{The behaviour
of Fenchel-Nielsen distance under a change of pants decomposition}

\author{D. Alessandrini}
\address{Daniele Alessandrini,  Max-Plank-Institut f\"ur Mathematik, Vivatsgasse 7, 53111 Bonn, Germany}
\email{daniele.alessandrini@gmail.com}

\author{L. Liu}
\address{Lixin Liu, Department of Mathematics, Sun Yat-Sen University, 510275, Guangzhou, P. R. China, and Hausdorff Research Institute for Mathematics, 
Poppelsdorfer Allee 45
D-53115 Bonn
Germany}
\email{mcsllx@mail.sysu.edu.cn}

\author{A. Papadopoulos}
\address{Athanase Papadopoulos,  Universit{\'e} de Strasbourg and CNRS,
7 rue Ren\'e Descartes,
 67084 Strasbourg Cedex, France, and Hausdorff Research Institute for Mathematics, 
Poppelsdorfer Allee 45
D-53115 Bonn
Germany} \email{athanase.papadopoulos@math.unistra.fr}
\date{\today}

\author{W. Su}
\address{Weixu Su, Department of Mathematics, Sun Yat-Sen University, 510275, Guangzhou, P. R. China, and Hausdorff Research Institute for Mathematics, 
Poppelsdorfer Allee 45
D-53115 Bonn
Germany}
\email{su023411040@163.com}

\begin{document}

\begin{abstract} Given a topological orientable surface of finite or infinite type equipped with a pair of pants decomposition $\mathcal{P}$ and given a base complex structure $X$ on $S$,  there is an associated deformation space of complex structures on $S$, which we call  the Fenchel-Nielsen Teichm\"uller space associated to the pair $(\mathcal{P},X)$. This space carries a metric, which we call the Fenchel-Nielsen metric, defined using Fenchel-Nielsen coordinates. We studied this metric in the papers \cite{ALPSS},
\cite{various} and  \cite{local}, and we compared it to the classical Teichm\"uller metric (defined using quasi-conformal mappings) and to another metric, namely, the length spectrum, defined using ratios of hyperbolic lengths of simple closed curves metric. In the present paper,
we show that under a change of pair of pants decomposition,
the identity map between the corresponding Fenchel-Nielsen metrics is not necessarily bi-Lipschitz. The results complement results obtained in the previous papers and they show that these previous results are optimal.
\end{abstract}

\maketitle

\bigskip


\medskip


\medskip
 \noindent L. Liu and W. Su are partially supported by NSFC grant No. 10871211.
\bigskip

\tableofcontents 

\section{Introduction}\label{intro}

This paper is in the lineage of the papers
\cite{ALPSS},
\cite{various},  \cite{local}
and \cite{LP}, in which we studied and compared various metrics on Teichm\"uller spaces of surfaces of finite or of  infinite topological type. 
The first important thing to know in that respect is that some definitions that are equivalent to each other in the Teichm\"uller theory of surfaces of finite type are no more equivalent in the setting of surfaces of infinite type. Indeed, there are Teichm\"uller spaces that are associated to a surface of infinite type that are distinct (in the set-theoretic sense) and that would be equal  if the same definitions were made in  the case of a surface of finite type. Furtermore, each such space associated to a surface of finite or of infinite type  carries a natural distance function, and it is an interesting problem to study the relations between the various spaces, their distance functions and their topologies. 

It is necessary to have different  names to the various spaces that arise, and we briefly recall the terminology.

We use the name \emph{quasi-conformal Teichm\"uller space} for the ``classical'' Teichm\"uller space defined using quasi-conformal mappings and equipped with its Teichm\"uller metric.

In the paper \cite{ALPSS}, we introduced the \emph{Fenchel-Nielsen Teichm\"uller space}, a certain space of equivalence classes of complex structures on a surface, equipped with a distance, called the \emph{Fenchel-Nielsen distance}, defined using Fenchel-Nielsen coordinates. This Teichm\"uller space and its metric depend on the choice of a
pair of pants decomposition of the surface. The Fenchel-Nielsen Teichm\"uller space was a fundamental tool in our work, because it has explicit coordinates and an explicit distance function, and we used it in the other papers mentioned to describe and understand the other Teichm\"uller spaces. One of the results obtained was that if the lengths of all the boundary curves of the pair of pants decomposition are bounded above by a uniform constant then there is a set-theoretic equality between the Fenchel-Nielsen Teichm\"uller space and the quasiconformal  Teichm\"uller space. Furthermore, the identity map between the two Teichm\"uller spaces, equipped with their respective metrics, is a locally bi-Lipschitz homeomorphism. This gives an explicit description of the global topology and the local metric properties of the quasiconformal Teichm\"uller space.

In the paper \cite{local}, we obtained similar local comparison results between the 
Fenchel-Nielsen Teichm\"uller space and the so-called \emph{length spectrum Teichm\"uller space},  another deformation space of complex structures, whose definition and metric are based on the comparison of   lengths of simple closed curves between surfaces.

In the cases of surfaces of finite type, the various Teichm\"uller spaces coincide set-theretically, but there are still interesting questions on the local and global comparison of the metrics that are defined on these spaces.

In the present paper, we prove that under a change of the pair of pants decomposition,
the identity map between the corresponding Fenchel-Nielsen Teichm\"uller spaces is not bi-Lipschitz in general. This result holds for surfaces of finite and for those of infinite type. We first prove this result in the case where the surface is a torus with one hole or a sphere with four holes. In this case, the two pair of pants decompositions are obtained from each other by a single elementary move. The proof is based on explicit computations that use formulae obtained by Okai in \cite{Okai}. We then deduce an analogous result for arbitrary surfaces of finite or of infinite type.

 To state the theorems precisely, we now introduce some minimal amount of notation. We refer the reader to Section \ref{section:preliminaries} for more details.
 
  If $X_0$ is a surface equipped with a complex structure, we denote by $\mathcal{T}_{qc}(X_0)$ the quasi-conformal Teichm\"uller space of $X_0$, equipped with the quasi-conformal distance $d_{qc}$, and by $\mathcal{T}_{FN, \mathcal{P}}(X_0)$ the Fenchel-Nielsen Teichm\"uller space of $X_0$ with reference to the pair of pants decomposition $\mathcal{P}$, equipped with its associated Fenchel-Nielsen distance $d_{FN,\mathcal{P}}$. 
  
  In Section \ref{section:general_case} we prove the following:

\begin{theorem}         \label{thm:main}
Let $S$ be an orientable surface which is either of finite topological type of negative Euler characteristic and which is not a pair of pants, or of infinite topological type. Let $\mathcal{P}$ be a pair of pants decomposition of $S$. Then we have the following:
\begin{enumerate}
\item There exists another pair of pants decomposition $\mathcal{P'}$ such that for every base complex structure $X_0$ on $S$ we have $\mathcal{T}_{FN,\mathcal{P}}(X_0) = \mathcal{T}_{FN,\mathcal{P'}}(X_0)$ as sets, but the identity map between the two spaces equipped with their
 Fenchel-Nielsen distances $d_{FN,\mathcal{P}}$ and $d_{FN,\mathcal{P'}}$ respectively is not Lipschitz.
\item For every base complex structure $X_0$ on $S$, consider the space $T = \mathcal{T}_{qc}(X_0) \cap \mathcal{T}_{FN,\mathcal{P}}(X_0)$. Then the identity map from $(T,d_{FN,\mathcal{P}})$ to $(T, d_{qc})$ is not Lipschitz.
\item If $S$ is of infinite topological type, there exists a base complex structure $X_0$ on $S$ and another pair of pants decomposition $\mathcal{P'}$ such that if $T$ is the space $\mathcal{T}_{qc}(X_0) \cap \mathcal{T}_{FN,\mathcal{P}}(X_0)$, then the identity map from $(T,d_{FN,\mathcal{P}})$ to $(T, d_{FN,\mathcal{P'}})$ is not continuous.
\item If $S$ is of infinite topological type, there exists a base complex structure $X_0$ on $S$ such that if $T$ is the space $\mathcal{T}_{qc}(X_0) \cap \mathcal{T}_{FN,\mathcal{P}}(X_0)$, then the identity map from $(T,d_{FN,\mathcal{P}})$ to $(T, d_{qc})$ is not continuous.  
\end{enumerate}
\end{theorem}

 From this result, we can see that the local comparison results we obtained in the papers \cite{ALPSS} and \cite{local} are optimal in the sense that they cannot be extended to global comparison results, since the global metric geometry of the Fenchel-Nielsen distance depends on the choice of the pair of pants decomposition.

\section{The Fenchel-Nielsen  metric}   \label{section:preliminaries}

In this section, we recall a few facts from our previous papers, which will help making the present paper self-contained.

Let $S$ be an orientable connected surface of finite or of infinite topological type, that can have punctures and/or  compact boundary components. The complex structures that we consider on $S$ are such that each boundary component has a regular neighborhood that is bi-holomorphically equivalent to a bounded cylinder and each puncture has a neighborhood that is bi-holomorphically equivalent to a punctured disc.

We start by reviewing the definition of the Fenchel-Nielsen metric on the Teichm\"uller space $\mathcal{T}(S)$ of $S$. The definition depends on the choice of a
pair of pants decomposition of $S$.

A pair of pants decomposition $\mathcal{P}=\{C_i\,\ i=1,2,\ldots\}$ of $S$ is a decomposition into generalized pairs of pants glued along their boundary components, where a generalized pair of pants is a sphere with three holes,  a hole  being either a point removed (leaving a puncture on the pair of pants) or an open disc removed (leaving a boundary component on the pair of pants). The curves $C_i$ in the above definition are the closed curves on $S$ (including the boundary components) that define the decomposition. It is well-known that every surface of finite topological type with negative Euler characteristic admits a pair of pants decomposition. It also follows from the classification of surfaces of infinite type that every such surface admits a pair of pants decomposition \cite{ALPSS}.

To every complex structure on $S$ we can associate a hyperbolic metric, called the \emph{intrinsic} metric, which was defined by Bers in \cite{Bers}. This metric is conformally equivalent to the given complex structure, and every boundary curve is a geodesic for that metric. The definition of the intrinsic metric is recalled in \cite{ALPSS}. In the sequel, when we talk about geometric objects (geodesics, length, angles, etc.) associated to a complex structure on $S$, it is understood that these are associated to the intrinsic metric on $S$. A pair of pants decomposition $\mathcal{P}=\{C_i\}$ of the surface $S$ equipped with its intrinsic metric is said to be \emph{geodesic} if each $C_i$ is a geodesic closed curve in $S$ with respect to this metric. In \cite{ALPSS} we proved that every topological pair of pants decomposition of $S$ is homotopic to a unique geodesic pair of pants decomposition. (Note that this is not true for general hyperbolic metrics on $S$. For example, the Poincar\'e metric of $S$ may not satisfy this property if this metric is different from the intrinsic metric, and this may happen;  we discussed this fact in \cite{ALPSS}.)

Given a closed curve $C$ on the surface $S$, we make the convention that we shall call also $C$ the unique geodesic representative of this closed curve with reference to the intrinsic metric.

Given a complex structure $X$ on $S$ and a geodesic pair of pants decomposition $\mathcal{P}=\{C_i\}$ of this surface, then for any closed geodesic $C_i\in \mathcal{P}$ there is a well-defined {\it length parameter} $l_X(C_i)$, which is the length of this closed geodesic, and a  {\it twist parameter} $\tau_X(C_i)$, which is defined only if $C_i$ is not the homotopy class of a boundary component of $S$. The quantity $\tau_X(C_i)$ is a measure of the relative twist amount along the geodesic $C_i$ between the two generalized hyperbolic pair of pants (which might be the same) having this geodesic in common. In this paper, the value $\tau_X(C_i)$ is a signed distance-parameter, that is, it represents an amount of twisting in terms of a distance measured along the curve, as opposed to an angle of twisting parameter (whose absolute value  would vary by $2\pi$ after a complete Dehn twist).

For any complex structure $X$ on $S$, its {\it Fenchel-Nielsen parameter} relative to $\mathcal{P}$ is a collection of pairs
\[\left((l_X(C_i),\tau_X(C_i))\right)_{i=1,2,\ldots}\]
where it is understood that if $C_i$ is a boundary geodesic, then it has no associated twist parameter, and instead of a pair $(l_X(C_i),\tau_X(C_i))$, we have a single parameter $l_X(C_i)$.

Given two complex structures $X$ and $X'$ on $S$, their {\it Fenchel-Nielsen distance} (with respect to $\mathcal{P}$) is defined as

\begin{equation}\label{def:FND}
d_{FN}(X,X')=\sup_{i=1,2,\ldots} \max\left(\left\vert \log \frac{l_X(C_i)}{l_{X'}(C_i)}\right\vert, \vert  \tau_X(C_i)-\tau_{X'}(C_i)\vert \right) ,
\end{equation}
again with the convention that if $C_i$ is a boundary component of $S$ and therefore has no associated twist parameter, we consider only the first factor.

 If $S$ is a surface of finite topological type, then $d_{FN}(X,X')$ is always finite, and the function $d_{FN}$ defines a distance on the Teichm\"uller space of $S$, which we will denote simply by $\mathcal{T}(S)$. If, instead, $S$ is of infinite topological type, $d_{FN}(X,X')$ can assume the value infinity. In this case we fix a base complex structure $X_0$ on $S$, and we define the Fenchel-Nielsen Teichm\"uller space of $X_0$ as the set of homotopy classes of complex structures $X$ on $S$ such that $d_{FN}(X_0, X)$ is finite. We denote this space by $\mathcal{T}_{FN}(X_0)$. The function $d_{FN}$ is a distance function in the usual sense on this space, and it makes it isometric to the sequence space $\ell^\infty$. 

 When it is important to stress on the dependence on a given pair of pants decomposition $\mathcal{P}$, we shall denote the Teichm\"uller space by $\mathcal{T}_{FN, \mathcal{P}}(X_0)$ (this dependence of the space on $\mathcal{P}$ may happen only for surfaces of infinite type), and the Fenchel-Nielsen distance by $d_{FN,\mathcal{P}}$ (the distance function depends on $\mathcal{P}$ also for a surface of finite type). 
 
 If a pair of pants decomposition  $\mathcal{P}$ is obtained from another pair of pants decomposition  $\mathcal{P}'$ by a finite number of elementary moves (represented in Figures \ref{torus} and \ref{sphere1} below), then for any basepoint $X_0$ we have the set-theoretic equality  $\mathcal{T}_{FN, \mathcal{P}}(X_0) = \mathcal{T}_{FN, \mathcal{P}'}(X_0)$. 
 
  In the last section of this paper we will prove that for every surface $S$, there exist two pair of pants decompositions $\mathcal{P}, \mathcal{P}'$ such that for every base complex structure $X_0$, the Fenchel-Nielsen Teichm\"uller spaces are the same (that is, we have a set-theoretical equality $\mathcal{T}_{FN, \mathcal{P}}(X_0) = \mathcal{T}_{FN, \mathcal{P}'}(X_0)$), but the identity map between the two spaces is not Lipschitz with reference to the two Fenchel-Nielsen distances $d_{FN,\mathcal{P}}$ and $d_{FN,\mathcal{P}'}$.

\section{The effect of an elementary move on the torus with one hole}\label{elementary-torus}

In this section, the surface $S=S_{1,1}$ is a torus with one hole, where the hole can be either a boundary component or a puncture.
A pair of pants decomposition of $S$ is determined by
a unique simple closed curve on $S$ which is not homotopic to a point or to the hole.

We consider two distinct pair of pants decomposition $\{\alpha\}$ and $\{\alpha'\}$ of $S$ defined by two essential simple closed curves $\alpha$ and $\alpha'$  satisfying
$i(\alpha,\alpha')=1$, as represented in Figure \ref{torus}.

  \begin{figure}[ht!]
\centering
\psfrag{a}{\small $\alpha$}
\psfrag{b}{\small $\alpha'$}
\psfrag{l}{\small $l_0$}
\psfrag{1}{\small (a)}
\psfrag{2}{\small (b)}
\includegraphics[width=.50\linewidth]{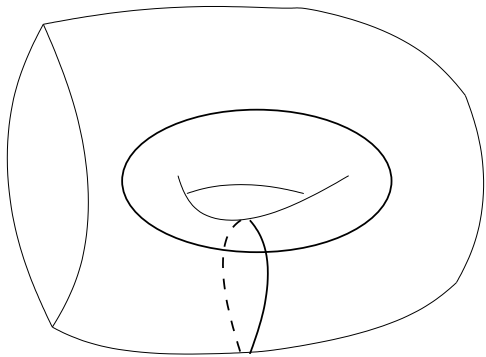}
\caption{\small {The two curves $\alpha$ and $\alpha'$ intersect at one point. An elementary move replaces one of these curves by the other one.}}
\label{torus}
\end{figure}

Let $X$ be a complex structure on $S$, equipped with its intrinsic hyperbolic metric. At the hole, $X$ can have geodesic boundary, and in this case we denote by $l_0$ its length, or it can have a cusp, and in this case we write $l_0 = 0$. We denote by $(l,\tau)$ the length and twist parameters of the curve $\alpha$, for the decomposition $\{\alpha\}$, and by $(l',\tau')$ the length and twist parameters of the curve $\alpha'$, for the decomposition $\{\alpha'\}$. We need a formula relating these values. This was done by Okai (\cite{Okai}) in the case where the hole is a boundary component, but we also need similar formulae for the case of a cusp. In the following proposition we obtain this with a continuity argument. In these formulae, the case where $l_0=0$ means that at the hole the surface has a cusp.

\begin{proposition}        \label{prop:Okaicusps}
With the above notation we have, for all $l_0\geq 0$,
$$\cosh(l'/2) = \sinh^{-1}(l/2)\cosh(\tau/2){\left(\frac{\cosh(l)+\cosh(l_0/2)}{2}  \right)}^{1/2} $$
$$\cosh(\tau'/2) = \cosh(l/2) {\left\{ \cosh^2(\tau/2) (\cosh(l)+\cosh(l_0/2)) - 2\sinh^2(l/2) \right\}}^{1/2} $$ 
$${\left\{ \cosh^2(\tau/2) (\cosh^2(l/2)+\cosh^2(l_0/2)) + \sinh^2(l/2)\cosh(l_0/2)  \right\}}^{1/2} .$$
\end{proposition}
\begin{proof}
When $l_0 > 0$, these formulae are proved in \cite{Okai}. To see that they also hold when $l_0 = 0$ (the case of a  cusp) we use the shear coordinates relative to an ideal triangulation.
We recall that in the case of a torus with one hole equipped with a complex structure (and the corresponding intrinsic metric), an ideal triangulation has three edges, say $a,b,c$, and associated to every edge there is a real number, called the shear parameter. Thus,  we have three numbers, $s_a,s_b,s_c$. The boundary length $l_0$ can be read from the shear coordinates:  $l_0 = |s_a+s_b+s_c|$ (see \cite[Prop. 3.4.21]{Thurston}), and the cusp corresponds to the case where $s_a+s_b+s_c=0$. Here we need the property that the parameters $l, |\tau|, l', |\tau'|$ can also be written as continuous functions of the shear coordinates. To see this one can show that for every element $\gamma$ of the fundamental group of the surface, the coefficients of the matrix corresponding to $\gamma$ in the holonomy representation can be written as continuous functions of the shear coordinates. The length of the corresponding closed geodesic in the surface is a continuous function of the trace of this matrix, hence all simple closed curve lengths are continuous functions of the shear coordinates. For the twist parameters, we note that their absolute values can be written as a continuous function of the lengths of some simple closed curves.

To show that the above formulae are valid for $l_0 = 0$, we just need to note that both the left hand side and the right hand side of the equations are continuous functions of the shear coordinates. It is known that shear coordinates extend to the case of surfaces with cusps (see  \cite[Chapter 3]{Thurston}), see also \cite{PT}).
Then as these functions are equal when $l_0 > 0$, i.e. when $s_a+s_b+s_c \neq 0$, and as this subset is dense in the space of shear coordinates, these functions  are equal everywhere, that is, including the case where $l_0 = 0$.      
\end{proof}

For simplicity, we choose a complex structure $X$ on the torus with one hole such that its intrinsic metric has the property that the two closed geodesics in the homotopy classes of $\alpha$
and $\alpha'$ meet perpendicularly. To see that such a choice is possible, we can start with an arbitrary
complex structure on $S$ (equipped with its intrinsic metric) and we cut this surface along the closed geodesic $\alpha$; in this way $\alpha'$ is cut into an essential geodesic arc, which is homotopic to a unique geodesic arc $\beta$ that is perpendicular to the two boundary components of $Y\setminus \{\alpha\}$ that arise from cutting the surface along $\alpha$.
Then we glue back the two components in such a way that in the resulting surface the
two endpoints of $\beta$ match. We obtain a complex structure with the desired property. Such a complex structure corresponds to the case $\tau=0$ in the
  notation of Okai \cite{Okai}. 

Performing a Fenchel-Nielsen twist of
magnitude $t$ along $\alpha$, we obtain from $X$ a new complex structure $X^t$
that has Fenchel-Nielsen coordinates denoted by
$(l,\tau_t)$ in the coordinate system associated to $\{\alpha\}$ and $(l'_t,\tau'_t)$ in the
coordinate system associated to $\{\alpha'\}$. The coordinates of the complex structure $X$ that we have chosen are $(l,0)$ and $(l',0)$ in the bases $\{\alpha\}$ and $\{\alpha'\}$ respectively.

Since $X^t$ is obtained by a time-$t$ twist along the curve $\alpha$, we have $\tau_t=t$ for all $t$.

We need to estimate $l'_t$ and $\tau'_t$.

By the formula for length in Proposition \ref{prop:Okaicusps}, we have
$$\cosh(l'/2)=\sinh^{-1}(l/2)\left(\frac{\cosh(l)+\cosh(l_0/2)}{2}\right)^{1/2}$$
and
$$\cosh(l'_t/2)=\sinh^{-1}(l/2)\cosh(t/2)\left(\frac{\cosh(l)+\cosh(l_0/2)}{2}\right)^{1/2}.$$

Thus, we have 
\[
 \cosh(t/2)=
\frac{\cosh(l'_t/2)}{\cosh(l'/2)}=\frac{e^{l'_t/2-l'/2}+e^{-l'_t/2-l'/2}}{1+e^{-l'}}.
\]
Using the fact that $x\leq e^x$, we obtain
\[0\leq l'_t/2-l'/2\leq e^{l'_t/2-l'/2}\leq \cosh(t/2)(1+e^{-l'})\]
which implies
\[0 < l'_t/2 \leq l'/2+ \cosh(t/2)(1+e^{-l'}).\]
As a result, and using $l'_t\geq l'$ (which follows from our hypothesis that $\alpha$ and $\alpha'$ meet perpendicularly for $t=0$), we have
\begin{eqnarray*}
0\leq \log\frac{l'_t}{l'} &\leq& \log\frac{l'/2+\cosh(t/2)(1+e^{-l'})}{l'/2} \\
&=&\log \left( 1+\frac{2\cosh(t/2)(1+e^{-l'})}{l'}\right) \\
&\leq& \frac{2\cosh(t/2)(1+e^{-l'})}{l'}.
\end{eqnarray*}

Now we estimate the twist parameter, using the formula for the twist in Proposition \ref{prop:Okaicusps}. For the purpose of the computations, we use the following notation:
$$A=\cosh^2 (l/2) \left( \cosh^2(t/2) \left( \cosh(l)+\cosh(l_0/2)\right) - 2\sinh^2(l/2) \right)$$
and
$$B= \cosh^2(t/2)\left(\cosh^2(l/2) +\cosh (l_0/2)\right)
 +\sinh^2 (l/2)\cosh(l_0/2).$$
 
The quantities $A$ and $B$ are functions of $t$ and $l$, and we have
$$\cosh^2(\tau'_t/2)= \frac{A}{B}.$$
We shall focus on the case where $l$ is small and $t,l_0$ are bounded.

Consider the second order expansions near $l=0$.

$$\cosh^2 (l/2)=1+l^2/4+O(l^4),$$
$$\cosh (l)=1+l^2/2+O(l^4),$$
$$\sinh^2 (l/2)=l^2/4+O(l^4).$$

Then we have
 
\begin{eqnarray*}
A&=& \left(1+l^2/4+O(l^4)\right)\times  \\
&\times&
\left(\left( \cosh^2(t/2)
\left( 1+l^2/2+O(l^4)+\cosh(l_0/2)\right) - \left(l^2/4+O(l^4)\right) \right)\right) \\
&=&\cosh^2(t/2)\left(1+\cosh(l_0/2)\right)\\&+&
\left(3\cosh^2(t/2)-2+\cosh^2(t/2)\cosh(l_0/2) \right)\frac{l^2}{4}+
 O(l^4)
\end{eqnarray*}

and

\begin{eqnarray*}
B&=&
  \cosh^2(t/2)  \left(1+l^2/4+O(l^4) +\cosh (l_0/2)\right)
 +\left(l^2/4+O(l^4)\right) \cosh(l_0/2) \\
 &=&\cosh^2(t/2)\left(1+\cosh(l_0/2)\right)+\left(\cosh^2(t/2)+\cosh(l_0/2) \right)\frac{l^2}{4}+O(l^4).
 \end{eqnarray*}

This gives
 
\begin{eqnarray*}
A-B &=& \left( 2\cosh^2(t/2)-2+\cosh^2(t/2)\cosh(l_0/2)-\cosh(l_0/2) \right)\frac{l^2}{4}+O(l^4)\\
 &=& \frac{1}{4} (2 + \cosh(l_0 / 2))\sinh^2(t/2) l^2 + O(l^4).
\end{eqnarray*}

 Note that the term $O(l^4)$ in the previous formula depends on $l_0$ and $|t|$ only via continuous functions of $l_0,|t|$. In particular, there exists an $\epsilon > 0$ such that for $l < \epsilon$ we have
\[
A-B\leq  Cl^2\]
where $C$ is a constant that depends only on $\epsilon$ and the upper bound of $l_0, |t|$.

Note that 
 $A\geq B>1$. From this and the properties of the logarithm function we obtain
 
\[\log \cosh^2(\tau'_t/2) = \log \frac{A}{B}= \log A -\log B \leq A-B.\]
Thus, we have 
 \[\log \cosh^2(\tau'_t/2) \leq Cl^2.\]

From the continuity of the $\cosh$ function, if $|\tau'_t|$ is bounded, then there is a constant $K$ that depends only on the upper bound for $|\tau'_t|$ such that $|\tau'_t|\leq K \log \cosh(\tau'_t/2)$.
Therefore,
$$|\tau'_t|\leq Ml$$
where $M$ is a constant that depends only on the upper bound of $l, l_0, |t|$.

We record the above results for the length and twist parameters in the following proposition.
\begin{proposition}\label{propositon:1-hold}
Let $X$ be a complex structure on $S_{1,1}$ with the following properties:
\begin{enumerate}
\item $X$ has either geodesic boundary of length $l_0$, or a cusp (and in the latter case we write $l_0 = 0$);
\item there exist a pair of perpendicular simple closed geodesics $\alpha, \alpha'$ on $X$ with $i(\alpha,\alpha')=1$.
\end{enumerate}
Let $X^t$ be the complex structure obtained from $X$ by performing a Fenchel-Nielsen twist of
magnitude $t$ along $\alpha$. Let $(l,\tau_t)$ and $(l'_t,\tau'_t)$ be the 
  Fenchel-Nielsen coordinates of $X^t$ 
in the coordinate systems associated to $\{\alpha\}$  and $\{\alpha'\}$ respectively. (Note that $l'_0= l'$.)

Assume $l_0$ and $|t|$ are bounded above by some constant $L>0$. Then there exist constants $M$ and $\epsilon_0>0$, both depending only on $L$, such that for all $l\leq \epsilon_0$, we have
 \begin{equation}\label{eq:col}
 \log\frac{ l'_t}{l'}\leq \frac{2\cosh(t/2)(1+e^{-l'})}{l'}\leq \frac{4\cosh (t/2)}{\vert \log l\vert}
 \end{equation}
 and
\[|\tau'_t|\leq Ml.\]
\end{proposition}

Note that the second inequality in (\ref{eq:col}) follows from one version of the Collar Lemma, which says that  there exists $\epsilon >0$ such that for $l\leq \epsilon_0$, we have $l'\geq \vert \log l\vert$.

\section{The effect of an elementary move on the sphere with four holes}\label{elementary-sphere}

In this section, the surface $S=S_{0,4}$ is the sphere with four holes, where each hole can be either a boundary component or a cusp. We equip $S$ with two pair of pants decompositions $\{\alpha\}, \{\alpha'\}$, defined by two essential simple closed geodesics satisfying $i(\alpha,\alpha')=2$. 

Let $X$ be a complex structure on $S$ equipped with its intrinsic metric. Near each of the four holes, $X$ can have a geodesic boundary, or a cusp. We denote by $l_1, l_2, l_3, l_4$ the lengths of the boundary components (as before with the convention that $l_i$ is zero if the corresponding hole is a cusp); see Figure \ref{sphere1}. We denote by $l,\tau$ the length and twist parameter respectively of the curve $\alpha$, for the decomposition $\{\alpha\}$, and by $l',\tau'$ the length and twist parameter respectively of the curve $\alpha'$, for the decomposition $\{\alpha'\}$. We need a formula relating these values. Like in the case of the torus with one hole, Okai wrote such formulae in \cite{Okai} in the  case where all the holes are boundary components, and we need to see that the formulae also hold in the case where some of the holes are cusps. In the following proposition we deduce this by a continuity argument, as we did in the previous section for the case of the one-holed torus.

\begin{proposition}        \label{prop:Okaicusps2}
With the above notation we have
$$\cosh(l'/2) = \sinh^{-2}(l/2) \{   \cosh(l_1/2)\cosh(l_2/2) + \cosh(l_3/2)\cosh(l_4/2)  $$
$$+ \cosh(l/2) \left[ \cosh(l_1/2)\cosh(l_3/2) + \cosh(l_2/2)\cosh(l_4/2) \right] + \cosh(\tau) [ \cosh^2(l/2)  $$ 
$$ + 2\cosh(l_1/2) \cosh(l_4/2) \cosh(l/2) + \cosh^2(l_1/2) + \cosh^2(l_4/2) - 1 ]^{1/2} [ \cosh^2(l/2) $$ 
$$ + 2\cosh(l_2/2) \cosh(l_3/2) \cosh(l/2) + \cosh^2(l_2/2) + \cosh^2(l_3/2) - 1   ]^{1/2} \} $$
and
$$ \cosh(\tau') = \{  \cosh^2(l_1/2) + \cosh^2(l_2/2) + 2 \cosh(l_1/2)\cosh(l_2/2)\cosh(l'/2) $$ 
$$+ \sinh^2(l'/2)   \}^{-1/2}  \{    \cosh^2(l_3/2) + \cosh^2(l_4/2) + 2 \cosh(l_3/2)\cosh(l_4/2)\cosh(l'/2) $$ 
$$+ \sinh^2(l'/2)   \}^{-1/2}  \{  \sinh^2(l'/2)\cosh(l/2) - \cosh(l_1/2)\cosh(l_4/2) - \cosh(l_2/2) \cosh(l_3/2)$$
$$ - \cosh(l'/2) {[\cosh(l_1/2)\cosh(l_3/2) + \cosh(l_2/2)\cosh(l_4/2)]} \}  $$
\end{proposition}
\begin{proof}
When all the boundary lengths are positive ($l_1,l_2,l_3,l_4 > 0$) these formulae are proved in \cite{Okai}. To see that they also hold when some $l_i$ is zero, we proceed as in the proof of Proposition \ref{prop:Okaicusps}. We fix an ideal triangulation, now having $6$ edges, so we have $6$ shear coordinates. For $i=1\dots 4$ the length $l_i$ can be expressed as the absolute value of the sum of the shear coordinates relative to the edges adjacent to the hole involved (see \cite[Prop. 3.4.21]{Thurston}), and the cusp corresponds to the case where this sum is zero. With exactly the same argument as in Proposition \ref{prop:Okaicusps} we can see that the parameters $l, |\tau|, l', |\tau'|$ can be written as continuous functions of the shear coordinates. Hence both the left hand side and the right hand side of the equations are continuous functions of the shear coordinates. As these functions are equal when $l_1,l_2,l_3,l_4 > 0$, and as this subset is dense in the space of shear coordinates, the functions are equal everywhere, including on the set where the values $l_i$ are zero.      
\end{proof}

Like in the case of the torus with one hole, we choose $X$ such that $\alpha$ and $\alpha'$ intersect perpendicularly, in order to simplify the computations. 
Performing a Fenchel-Nielsen twist of
magnitude $t$ along $\alpha$, we obtain from $X$ a new complex structure $X^t$
and we denote its Fenchel-Nielsen coordinates
$(l,\tau_t)$ in the coordinate system associated to $\{\alpha\}$ and $(l'_t,\tau'_t)$ in the
coordinate system associated to $\{\alpha'\}$. As before we have $\tau_t = t$. Then we also have the following proposition.

  \begin{figure}[ht!]
\centering
\psfrag{a}{\small $l_1$}
\psfrag{b}{\small $\alpha'$}
\psfrag{l}{\small $\alpha$}
\psfrag{m}{\small $l_2$}
\psfrag{n}{\small $l_3$}
\psfrag{o}{\small $l_4$}
\includegraphics[width=.45\linewidth]{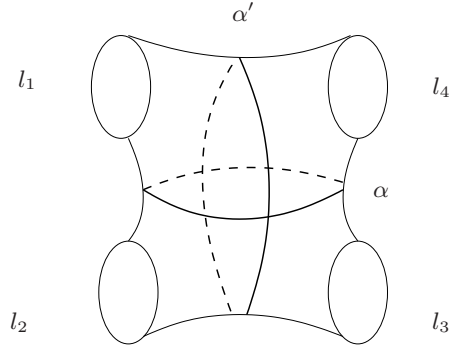}
\caption{\small {An elementary move on the sphere with four holes replaces one of the interior curves drawn by the other one.}}
\label{sphere1}
\end{figure}

\begin{proposition}\label{proposition:4-holds}
Let $X$ be a complex structure on $S_{0,4}$ with the following properties:
\begin{enumerate}
\item every hole of $X$ is a geodesic boundary component or a cusp, with lengths denoted by $l_1, l_2, l_3, l_4$ (with the convention that $l_i$ is zero if the corresponding hole is a cusp);
\item there exist a pair of perpendicular simple closed geodesics $\alpha, \alpha'$ on $X$ with $i(\alpha,\alpha')=2$.
\end{enumerate}
Let $X^t$ be the complex structure obtained from $X$ by performing a Fenchel-Nielsen twist of
magnitude $t$ along $\alpha$. Let $(l,\tau_t)$ and $(l'_t,\tau'_t)$ be the Fenchel-Nielsen coordinates of $X^t$ 
in the coordinate systems associated to $\{\alpha\}$  and $\{\alpha'\}$ respectively. (Note that $l'_0= l'$.)

Then, 
if $l_1,l_2,l_3,l_4$  and $|t|$ are bounded from above by some constant $L$, there exist constants $K, M, \epsilon_0>0$, all of them depending only on $L$, such that for $l \leq \epsilon_0$ we have
\begin{equation}\label{eq:log}
\log\frac{ l'_t}{l'}\leq \frac{K(1+e^{-l'})}{l'}\leq \frac{2K}{\vert \log l\vert}.
\end{equation}
and
\[|\tau'_t|\leq Ml.\]
\end{proposition}
\begin{proof}
We note as in the case studied before that the second inequality in (\ref{eq:log}) follows from the Collar Lemma.

To prove the proposition, we use the formulae in Proposition \ref{prop:Okaicusps2}. Setting

$$A =  \cosh(l_1/2)\cosh(l_2/2) + \cosh(l_3/2)\cosh(l_4/2)$$ 
$$ + \cosh(l/2) \left[ \cosh(l_1/2)\cosh(l_3/2) + \cosh(l_2/2)\cosh(l_4/2) \right]$$

and

$$ B = [ \cosh^2(l/2)  + 2\cosh(l_1/2) \cosh(l_4/2) \cosh(l/2) + \cosh^2(l_1/2) + \cosh^2(l_4/2) - 1 ]^{1/2}$$
$$ [ \cosh^2(l/2)  + 2\cosh(l_2/2) \cosh(l_3/2) \cosh(l/2) + \cosh^2(l_2/2) + \cosh^2(l_3/2) - 1   ]^{1/2},$$
the formula for the length parameter $l'$ becomes

$$\cosh(l'_t/2)=\sinh^{-2}(l/2)\left(A+\cosh(t)B\right).$$

  For $t=0$, this formula becomes
  $$\cosh(l'/2)=\sinh^{-2}(l/2)\left(A+B\right).$$

Thus, we have 
$$\frac{\cosh(l'_t/2)}{\cosh(l'/2)}=\frac{\left(A+\cosh(t)B\right)}{\left(A+B\right)}\leq K,$$
where $K$ is a constant that depends only on $L$.
Using the same estimates as in Section \ref{elementary-torus}, we 
obtain 
 $$\log \frac{l_t'}{l'}\leq \frac{2K(1+e^{-l'})}{l'}.$$

Now we estimate the twist $\tau'_t$. We assume that the twist is positive, to avoid taking absolute values. The formula in Proposition \ref{prop:Okaicusps2} gives 
\begin{equation}\label{equ:tau2}
\cosh(\tau'_t)=\frac{E}{F},
\end{equation}
where

$$ E = \sinh^2(l_t'/2)\cosh(l/2) - \cosh(l_1/2)\cosh(l_4/2) - \cosh(l_2/2) \cosh(l_3/2)$$ 
$$ - \cosh(l_t'/2) {[\cosh(l_1/2)\cosh(l_3/2) + \cosh(l_2/2)\cosh(l_4/2)]}$$

$$ F = \{  \cosh^2(l_1/2) + \cosh^2(l_2/2) + 2 \cosh(l_1/2)\cosh(l_2/2)\cosh(l_t'/2) + \sinh^2(l_t'/2)   \}^{1/2} $$
$$ \{    \cosh^2(l_3/2) + \cosh^2(l_4/2) + 2 \cosh(l_3/2)\cosh(l_4/2)\cosh(l_t'/2) + \sinh^2(l_t'/2)   \}^{1/2} . $$
 
\vskip0.2cm

Note that $E>F$ for all $t\not= 0$. Let $E_1 = E / \sinh^2(l_t'/2)$, and $F_1 = F / \sinh^2(l_t'/2)$.

As $l\to 0$, we have $l'\to \infty$, and $E_1,F_1 \to 1$. As a result, we can assume that $l$ is sufficiently small, so that 
$E_1$ and $F_1$ are not less than $1/2$.
Then 
$$\log \cosh(\tau'_t)=\log \frac{E}{F}=\log \frac{E_1}{F_1}\leq 2(E_1-F_1).$$
(We used the fact that $\log x -\log y\leq 2(x-y)$ for $\frac{1}{2}\leq x\leq y$.)

Now we need to estimate $E_1 - F_1$. To do this, note that by the first formula of Proposition \ref{prop:Okaicusps}, we have
$$\cosh(l_t'/2) \geq 8\sinh^{-2}(l/2).$$
Using the fact that $\sinh^2(x) = \cosh^2(x) - 1$ we conclude that if $l$ is small, there exists a constant $C$ depending only on $\epsilon_0$ such that
$$\sinh(l_t'/2) \geq \frac{C}{l^2}.$$
First we estimate $E_1 - \cosh(l/2)$:
$$|E_1 - \cosh(l/2)| = | \frac{- \cosh(l_1/2)\cosh(l_4/2) - \cosh(l_2/2) \cosh(l_3/2)}{\sinh^2(l_t'/2)}$$ 
$$ - \frac{\cosh(l_t'/2) {[\cosh(l_1/2)\cosh(l_3/2) + \cosh(l_2/2)\cosh(l_4/2)]} }{\sinh^2(l_t'/2)} | \leq H l^2$$
where $H$ is a constant depending only on $L$ and $\epsilon_0$.

Then we estimate $F_1 - 1$:
$$|F_1 - 1| = |\{ 1 + \frac{ \cosh^2(l_1/2) + \cosh^2(l_2/2) + 2 \cosh(l_1/2)\cosh(l_2/2)\cosh(l_t'/2)}{\sinh^2(l_t'/2)}  \}^{1/2} $$ 
$$ \{ 1 + \frac{   \cosh^2(l_3/2) + \cosh^2(l_4/2) + 2 \cosh(l_3/2)\cosh(l_4/2)\cosh(l_t'/2)}{\sinh^2(l_t'/2) }  \}^{1/2} - 1 | $$
$$= |( 1 + R l^2 + O(l^4)) (1 + R l^2 + O(l^4)) -1| \leq 2 S l^2 $$
where $R$ and $S$ are constants depending only on $L$ and $\epsilon_0$. 

Then, finally:

\begin{eqnarray*}
|E_1 - F_1| & \leq&  |E_1 - \cosh(l/2)| + |F_1 - 1| + |\cosh(l/2) - 1|\\
&  \leq&  H l^2 + S L^2 + \frac{l^2}{2}+0(l^4) .
\end{eqnarray*}

Now if $|\tau'_t|$ is small, then there is a constant $K$ such that $(\tau'_t)^2 \leq K \log \cosh(\tau'_t)$.
(We use the formula $\log \cosh x= x^2/2+ O(x^4)$.)
Therefore,
$$|\tau'_t|\leq Ml.$$
This proves Proposition \ref{proposition:4-holds}.
\end{proof}

\section{Comparing Fenchel-Nielsen distances}\label{surfaces-finite}

We reformulate Proposition \ref{propositon:1-hold} and Proposition \ref{proposition:4-holds} in the following.

\begin{proposition}\label{prop:length-twist}
Let $X$ be a complex structure on $S_{1,1}$ or $S_{0,4}$ that satisfies the assumptions in 
Proposition \ref{propositon:1-hold} or Proposition \ref{proposition:4-holds} respectively.  
Assume that  $|t|$, 
$l_0$ (in the case where $S=S_{1,1}$) and $l_1,l_2,l_3,l_4$ (in the case where $S=S_{0,4}$) are bounded above by some constant $L$. Then
 there exist constants $K$, $M$ and $\epsilon_0>0$, all depending only on $L$ such that for $l\leq \epsilon_0$, we have 
$$\log\frac{ l'_t}{l'}\leq \frac {K}{|\log l|}$$
and
$$|\tau'_t|\leq Ml.$$
\end{proposition}

 For $S=S_{1,1}$ or $S_{0,4}$, let $d_{FN_1}$ and $d_{FN_2}$ denote the Fenchel-Nielsen coordinates on $\mathcal{T}(S)$ associated to $\{\alpha\}$ and $\{\alpha'\}$ respectively. As above, for any complex structure $X$ on $S$, we denote by $X^t$ the complex structure obtained by the time-$t$ Fenchel-Nielsen twist of $X$ along $\alpha$.
We have the following:
\begin{proposition} \label{prop:not-Lipschitz1}
Suppose that $S$ is either the torus with one hole $S_{1,1}$ or the sphere with four holes  $S_{0,4}$, with $\alpha$ and $\alpha'$ being homotopy classes of closed curves in $S$ satisfying $i(\alpha,\alpha')=1$ in the case $S=S_{1,1}$ and $i(\alpha,\alpha')=2$ in the case $S=S_{0,4}$.
Then, there exist a sequence of points $X_n\in \mathcal{T}(S)$ such that
$$d_{FN_1}(X_n,X_n^t)=|t|, \ \mathrm{while} \ \lim_{i\to\infty}d_{FN_2}(X_n,X_n^t)=0.$$
\end{proposition}

\begin{proof} 
Let $X_n$, $n=1,2,\ldots$ be a sequence of complex structures on $S$ satisfying the following properties:
\begin{enumerate} 
\item for any $n=1,2,\ldots$ the geodesics in the classes of $\alpha$ and $\alpha'$ intersect perpendicularly;
\item the hyperbolic length $l_{X_n}(\alpha)=\epsilon_n$ tends to $0$ as $n\to\infty$;
\item the hyperbolic length of the holes of the torus with one hole $S_{1,1}$ or of the sphere with four holes $S_{0,4}$ is bounded by some fixed constant $L$.
\end{enumerate}
It is clear that $d_{FN_1}(X_n,X_n^t)=|t|$. By Proposition \ref{prop:length-twist}, we have 
$$0\leq \log\frac{ l'_t}{l'}\leq \frac {K}{|\log \epsilon_n|},$$
and
$$|\tau'_t|\leq M\epsilon_n.$$
As a result, 
$$d_{FN_2}(X_n,X_n^t)=\max\{|\log\frac{ l'_t}{l'}|,|\tau'_t|\}\to 0,$$
as $\epsilon_n\to 0$.
\end{proof}
We conclude with the following

\begin{corollary}\label{co:non-sphere} With the notation of Proposition \ref{prop:not-Lipschitz1}, the identity map between the metrics $d_{FN_1}$ and $d_{FN_2}$ is not Lipschitz. More precisely, there does not exist any constant $C$ satisfying $d_{FN_1}(x,y)\leq C d_{FN_2}(x,y)$ for all $x$ and $y$ in $\mathcal{T}(S)$.
\end{corollary}

 \section{General surfaces}    \label{section:general_case}
 
The aim of this section is to prove Theorem \ref{thm:main}, which is the analogue of Corollary \ref{co:non-sphere} for an arbitrary surface $S$ of finite or infinite type. 

The idea is to start with an arbitrary
pair of pants decomposition $\mathcal{P}=\{C_i\}$ of $S$, to take in it an embedded subsurface of type $S_{1,1}$ or $S_{0,4}$ with boundary curves belonging to the system $\{C_i\}$, and to modify the pair of pants decomposition $\mathcal{P}$ into a pair of pants decomposition $\mathcal{P}'$  by an elementary move $\alpha \to \alpha'$ performed inside this subsurface, according to the scheme used in Sections \ref{elementary-torus} and \ref{elementary-sphere}. 

There is one complication in doing this. Even though in the new pair of pants decomposition $\mathcal{P}'$ the length parameters of all the non-modified curves for the complex structure $X^t$ obtained by twisting along the curve $\alpha$ is the same in the systems $\mathcal{P}$ and $\mathcal{P}'$, the situation is not the same for the twist parameters. Performing the Fenchel-Nielsen twist along the curve $\alpha$ does not modify the twist parameters of the closed curves in the system $\mathcal{P}$ that are different from $\alpha$, but in the system  $\mathcal{P}'$, the twist parameters of the curves that are on the boundary of the subsurface $S_{1,1}$ or $S_{0,4}$ do not remain constant. This is the main question that we  deal with now.

Let $X$ be a complex structure on $S$ equipped with a 
geodesic pair of pants decomposition $\mathcal{P}$, and let $\alpha\in \mathcal{P}$ be a closed geodesic in the interior of $X$.

The closed geodesic $\alpha$ is either in the interior of a torus with one hole or of a sphere with four holes that is defined by the pair of pants decomposition  $\mathcal{P}$.
We denote such a one-holed torus or a four-holed sphere by $Y$.

Let $\mathcal{P}'$ be the pair of 
pants decomposition of $X$ obtained from $\mathcal{P}$ by an elementary move on $\alpha$, replacing this curve with a curve $\alpha'$ which is  contained in $Y$ and which has
minimal intersection number with $\alpha$.

  \begin{figure}[ht!]
\centering
\psfrag{1}{\small $C_4$}
\psfrag{2}{\small $C_3$}
\psfrag{3}{\small $\alpha$}
\psfrag{4}{\small $\alpha'$}
\psfrag{5}{\small $C_1$}
\psfrag{6}{\small $C_2$}
\psfrag{7}{\small $\beta$}
\includegraphics[width=.80\linewidth]{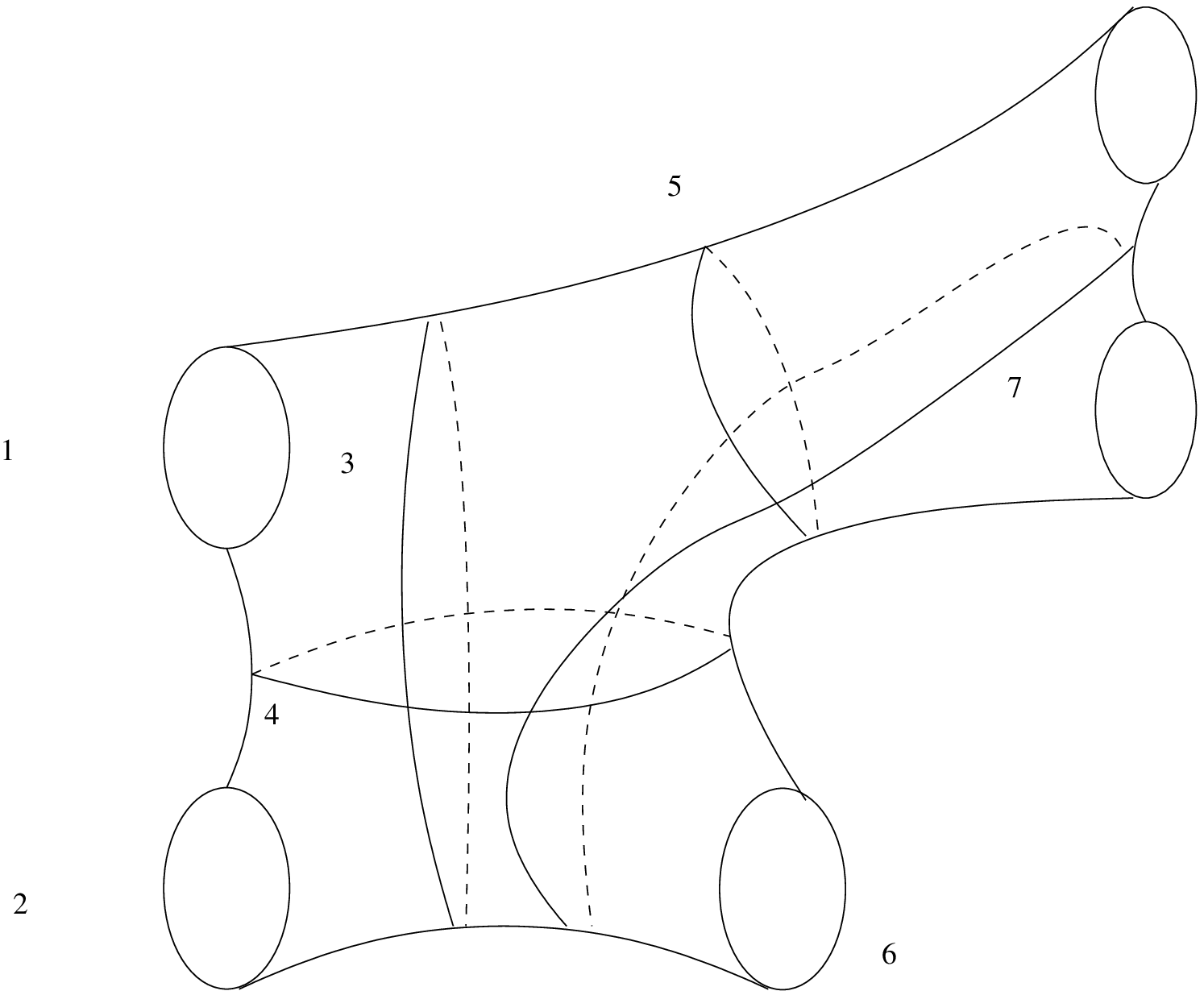}
\caption{\small{}}
\label{elementary}
\end{figure}

We take a sequence $X_i$ of complex structures on $S$ satisfying  $l_{X_{i}}(\alpha)\to 0$. By the collar lemma, we then have $l_{X_{i}}(\alpha')\to \infty$. 

We assume that the curve $\alpha$ is adjacent to two distinct pairs of pants in the decomposition $\mathcal{P}$. The other case can be dealt with in the same way.  We shall use the notation of Figure \ref{elementary}.

We shall apply Proposition \ref{proposition:4-holds}, and for this
we  assume that $X$ is a complex structure such that $l_X(\alpha)$ is sufficiently small.

Let $X^t$ be as before the time-$t$ twist of $X$ along $\alpha$. Then, $d_{FN}(X,X^t)=|t|$. 

To estimate
 the distance $d_{FN'}(X,X^t)$, first apply Proposition \ref{prop:length-twist}:
 $$\log\frac{ l_{X^t}(\alpha')}{l_{X}(\alpha')}\leq \frac {K}{|\log l_X(\alpha)|}, $$
 $$|\tau'_{X^t}(\alpha')-\tau'_{X}(\alpha')|\leq Ml_X(\alpha).$$
 For any $\beta\in \mathcal{P}',\beta\neq \alpha'$, the length is unchanged:
 $$\log\frac{ l_{X^t}(\beta)}{l_{X}(\beta)}=0.$$
 However, the equality
 $\tau'_{X^t}(\beta)= \tau'_{X}(\beta)$ is not necessarily true. Indeed, let $C_1,\ldots,C_4$ be the boundary components of the sphere with four holes that is the union of the two pairs of pants of $\mathcal{P}$ that contain $\alpha$ (notation of Figure \ref{elementary}).
 In the Fenchel-Nielsen coordinates of $\mathcal{P}'$, the length of $\alpha'$ is changed under the twist along $\alpha$, 
 and for each $j=1,\ldots,4$, $\tau'_{X^t}(C_j)$ may be not be equal to $\tau'_{X}(C_j)$.

However, we have the following:
\begin{lemma}\label{lemma:othertwist} For $i=1,\ldots,4$, 
 $|\tau'_{X^t}(C_i)-\tau'_{X}(C_i)|\to 0$, as $l_X(\alpha)\to 0$.
 \end{lemma}
\begin{proof}
For the proof, we take $i=1$.

Choose a simple closed curve $\beta$ as in Figure \ref{elementary}, contained in the 4-holed sphere that is the union of the pair of pants with boundary curves $\alpha',C_1,C_2$ and the other pair of pants adjacent to $C_1$,
such that $\beta$ satisfies $i(C_1, \beta)=2$ and is disjoint from any other
curve in $\mathcal{P}\setminus \{C_1\}$, as in Figure \ref{elementary}.

We may assume that the angles at the  intersection points of $C_1$ with $\beta$ satisfy
$\cos \theta\geq \frac{1}{2}$. To achieve this we can choose a sufficiently large integer $N$, and replace $\beta$ with its image under an $N$-order positive Dehn twist
along $C_1$.

We now apply the first variational formula of Wolpert \cite{Wolpert83}.

For $s>0$, we denote by $l_s(\beta)$ the hyperbolic length of $\beta$
under the twist of amount $s$ along $C_1$.  Denoting the intersection angles of
$C_1$ with $\beta$ by $\theta_{1,s}$ and $\theta_{2,s}$, we have
\[
\frac{dl_s(\beta)}{ds}=\cos \theta_{1,s}+ \cos \theta_{2,s}.
\]
By the mean value theorem, for each $s>0$, there exists a $\xi\in [0,s]$, such that
\begin{equation}\label{equation:wolpert}
l_s(\beta)-l_0(\beta)=( \cos \theta_{1,\xi}+ \cos \theta_{2,\xi}) s.
\end{equation}
Since $\cos \theta_{j,s},j=1,2$ are strictly increasing functions of $s$ (Proposition 3.5 of Kerckhoff \cite{Kerckhoff}),
we have
\begin{equation}\label{equation:cos}
\cos \theta_{i,\xi}\geq \cos\theta_{i,0}\geq \frac{1}{2},
\end{equation}
by assumption.  Combining (\ref{equation:wolpert}) and (\ref{equation:cos}), we have
\begin{equation}\label{equation:inequ}
l_s(\beta)-l_0(\beta)\geq s.
\end{equation}

We show that $\vert\tau'_{X^t}(C_1)-\tau'_{X}(C_1)\vert\to 0$ as $l_X(\alpha)\to 0$. Note that $l_X(C_1)=l_{X^t}(C_1)$ and $\beta$ also intersects $\alpha'$.

Suppose first that the length of
$\beta$ is decreased by an amount of $x_t>0$ under the
 twist of distance $t'=\tau'_{X^t}(\alpha')-\tau'_{X}(\alpha')$ along $\alpha'$. Then this length should be
increased by the same amount $x_t$ under the action of the twist along $C_1$.
Note that $x_t\leq 2|t'|$.

Applying the inequality (\ref{equation:inequ}), we have 
\begin{equation}\label{eq:in}
|\tau'_{X^t}(C_1)-\tau'_{X}(C_1)|\leq 2|t'|=2|\tau'_{X^t}(\alpha')-\tau'_{X}(\alpha')|.
\end{equation}
By Proposition \ref{prop:length-twist}, $t' \to 0$
as $l_X(\alpha)\to 0$. As a result, $|\tau'_{X^t}(C_1)-\tau'_{X}(C_1)| \to 0$ as $l_X(\alpha)\to 0$.

The case where the length of $\beta$ is increased under the twist along $\alpha'$ can be dealt with by the same argument.
\end{proof}

 Now we are ready to prove the various statements of Theorem \ref{thm:main}. 
 
 Let us summarize the setting. $S$ is an orientable surface which is either of infinite topological type 
 or of finite topological type with negative Euler characteristic and which is not homeomorphic to a pair of pants. Let $\mathcal{P}$ be a pair of pants decomposition of $S$. Let $\alpha \in \mathcal{P}$ be a curve which is not a boundary component of $S$, and consider the pair of pants decomposition $\mathcal{P}'$ obtained from $\mathcal{P}$ by an elementary move about $\alpha$. We denote by $\alpha'$ the curve replacing $\alpha$ in $\mathcal{P}'$. Let $Y$ be the one-holed torus or four-holed sphere containing $\alpha$ in its interior. If $X$ is a complex structure on $S$, and $t$  a real number, we denote by $X^t$ the structure obtained by a time-$t$ Fenchel-Nielsen twist of $X$ along $\alpha$.

\begin{theorem}         
For every base complex structure $X_0$ on $S$ we have $\mathcal{T}_{FN,\mathcal{P}}(X_0) = \mathcal{T}_{FN,\mathcal{P'}}(X_0)$ as sets, but the identity map between these two spaces is not Lipschitz with respect to the  metrics $d_{FN,\mathcal{P}}$ and $d_{FN,\mathcal{P'}}$ respectively. More precisely, there exists a sequence of points $X_n \in \mathcal{T}_{FN,\mathcal{P}}(X_0)$ such that 
$$d_{FN,\mathcal{P}}(X_n, X_n^t) = |t|, \mbox{ while } \lim_{n \to \infty} d_{FN,\mathcal{P}'}(X_n, X_n^t) = 0.$$
\end{theorem}
\begin{proof}
Let $X_n$, $n=1,2,\ldots$ be a sequence of complex structures on $S$ satisfying the following properties:
\begin{enumerate} 
\item for any $n=1,2,\ldots$ the geodesics in the classes of $\alpha$ and $\alpha'$ intersect perpendicularly;
\item the hyperbolic length $l_{X_n}(\alpha)=\epsilon_n$ tends to $0$ as $n\to\infty$;
\item the hyperbolic length of the boundary curves of the pairs of pants containing $\alpha$ and $\alpha'$ is bounded by some fixed constant $L$.
\end{enumerate}
It is clear that $d_{FN, \mathcal{P}}(X_n,X_n^t)=|t|$. By Proposition \ref{prop:length-twist}, there exists a constant $K$ such that for all $n$
$$|\log \frac{l_{X_n^t}(\alpha')}{l_{X_n}(\alpha')}| \leq \frac {K}{|\log \epsilon_n|}$$
and
$$|\tau'_{X_n^t}(\alpha')-\tau'_{X_n}(\alpha')|\leq K\epsilon_n.$$
It follows from Lemma \ref{lemma:othertwist}  that for each $C_j\in \mathcal{P}'\setminus \{\alpha\}$,
$$\lim_{n\to\infty}\sup_{C_j\in \mathcal{P}'\setminus \{\alpha\}}|\tau'_{X_n^t}(C_j)-\tau'_{X_n}(C_j)|=0.$$
As a result,
$$\lim_{n\to\infty}d_{FN, \mathcal{P}'}(X_n,X_n^t)= 0.$$
\end{proof}

\begin{theorem}
For every base complex structure $X_0$ on $S$, consider the space $T = \mathcal{T}_{qc}(X_0) \cap \mathcal{T}_{FN,\mathcal{P}'}(X_0)$. Then the identity map from $(T,d_{FN,\mathcal{P}'})$ to $(T, d_{qc})$ is not Lipschitz. More precisely, there exists a sequence of points $X_n \in T$ and a constant $C > 0$ such that 
$$d_{qc}(X_n, X_n^t) \geq C|t|, \mbox{ while } \lim_{n \to \infty} d_{FN,\mathcal{P}'}(X_n, X_n^t) = 0.$$
\end{theorem}

\begin{proof}
We use the same sequence as in the previous theorem. We claim that there exists a constant $C$, depending only on the constants $L$ and $|t|$, and such that $d_{qc}(X_n, X_n^t) \geq C d_{FN,\mathcal{P}}(X_n, X_n^t) = C|t|$. This essentially follows from Theorem 7.6 in \cite{ALPSS}, but some special care is needed because, in the form that theorem is stated, one of the hypotheses is not satisfied: that the surfaces $X_n$ are upper bounded, i.e. that there exists a constant $M$ such that for every curve $C \in \mathcal{P}$ we have $l_{X_0}(C) \leq M$. This is not  satisfied in general. The point is that here we are performing the twist on surfaces $X_n$ along the single curve $\alpha$ (while in Theorem 7.6 in \cite{ALPSS} we allowed a multi-twist around many curves), so we only need to check that the lengths of the curves of the pairs of pants containing $\alpha$ are bounded by a constant. This is true in this case, because we are assuming that the lengths are bounded above by $L$. To apply the theorem we also need to check that the lengths $d_{qc}(X_n, X_n^t)$ are bounded above by a constant that depends only on $L$ and $|t|$. This is given by Lemma 8.3 in \cite{ALPSS}, again by paying attention to the fact that, even if the surfaces are not upper bounded, we are twisting only around $\alpha$, and the lengths of the boundary curves of the pairs of pants containing $\alpha$ are bounded by $L$.  
\end{proof}

Now assume that $S$ is an orientable surface of infinite topological type and let $\mathcal{P}$ be a pair of pants decomposition of $S$. We choose a sequence of curves $\alpha_i \in \mathcal{P}$ such that the one-holed tori or four-holed spheres $Y_i$ containing $\alpha_i$ in their interior are all disjoint. We also consider another pair of pants decomposition $\mathcal{P}'$, obtained from $\mathcal{P}$ by performing elementary moves about all the curves $\alpha_i$. For every $i$, let $\alpha_i'$ be the curve replacing $\alpha_i$ in $\mathcal{P}'$. Finally, we choose a base complex structure $X_0$ on $S$ satisfying the following:
\begin{enumerate}
\item $l_{X_0}(\alpha_i) \to 0$ as $i\to\infty$;
\item for all $i$ the geodesics in the classes of $\alpha_i$ and $\alpha_i'$ intersect perpendicularly;
\item the lengths of all the curves of the decomposition $\mathcal{P}$ are bounded above by some global constant $M$. 
\end{enumerate}
For a real number $t$, we denote by $X_i$ the surface obtained by a time-$t$ Fenchel-Nielsen twist of $X_0$ along $\alpha_i$. 
 
We have the following:
\begin{theorem}
If $T$ is the space $\mathcal{T}_{qc}(X_0) \cap \mathcal{T}_{FN,\mathcal{P}}(X_0)$, then the identity map from $(T,d_{FN,\mathcal{P}'})$ to $(T, d_{FN,\mathcal{P}})$ is not continuous. More precisely, using the above notation, the surfaces $X_i$ are in $T$, and they satisfy
$$d_{FN,\mathcal{P}}(X_0, X_i) = |t|, \mbox{ while } \lim_{i \to \infty} d_{FN,\mathcal{P}'}(X_0, X_i) = 0.$$
\end{theorem}

\begin{proof}
Assume that $l_{X_0}(\alpha_i)=\epsilon_i \to 0$. It is clear that $d_{FN, \mathcal{P}}(X_0,X_i)=|t|$. Applying Proposition \ref{prop:length-twist}, we have
$$|\log \frac{l_{X_i}(\alpha_i')}{l_{X_0}(\alpha_i')}| \leq \frac {K}{|\log \epsilon_i|},$$
and
$$|\tau'_{X_i}(\alpha_i')-\tau'_{X_0}(\alpha_i')|\leq K\epsilon_i,$$
where $K$ is a constant depending on $M$.
It follows from Lemma \ref{lemma:othertwist}  that for each $C_j \in \mathcal{P}'\setminus \{\alpha_i\}$,
$$\lim_{i\to\infty}\sup_{C_j\in \mathcal{P}'\setminus \{\alpha_i\}}|\tau'_{X_i}(C_j)-\tau'_{X_0}(C_j)|=0.$$
As a result,
$$\lim_{i\to\infty}d_{FN, \mathcal{P}'}(X_0,X_i)= 0.$$
\end{proof}

\begin{theorem}
If $T$ is the space $\mathcal{T}_{qc}(X_0) \cap \mathcal{T}_{FN,\mathcal{P}'}(X_0)$, then the identity map from $(T,d_{FN,\mathcal{P}'})$ to $(T, d_{qc})$ is not continuous. More precisely, the surfaces $X_i$ defined above are in $T$, and there exists a constant $C > 0$ such that
$$d_{qc}(X_0, X_i) \geq C|t|, \mbox{ while } \lim_{i \to \infty} d_{FN,\mathcal{P}'}(X_0, X_i) = 0.$$  
\end{theorem}

\begin{proof}
As before, we claim that there exists a constant $C$, depending only on the constant $M$ and on $|t|$, such that $d_{qc}(X_i, X_i^t) \geq C d_{FN,\mathcal{P}}(X_i, X_i^t) = C|t|$. This follows from Theorem 7.6 in \cite{ALPSS}. This time all the hypotheses of that theorem are satisfied, and we only need to check that $d_{qc}(X_i, X_i^t)$ is bounded by something that depends only on $L$ and $|t|$. This is given by Lemma 8.3 in \cite{ALPSS}.  
\end{proof}

\end{document}